\documentclass[11pt,bezier]{article}
\usepackage{amsmath,amssymb,amsfonts,euscript,graphicx}

\newtheorem{preproof}{{\bf \indent Proof.}}

\newenvironment{proof}[1]{\begin{preproof}{\rm
               #1}\hfill{$\Box$}}{\end{preproof}}

\newtheorem{cor}{\bf\indent Corollary}[section]
\newtheorem{example}{\bf\indent Example}[section]
\newtheorem{thm}{{\bf\indent Theorem}}[section]

\newtheorem{remark}{\sc\indent Remark}[section]
\newtheorem{lem}{\bf\indent Lemma}[section]

\title{\bf \large Metric dimension in a prime ideal sum graph\\ of a commutative ring }

\author{{\normalsize   { {M. Adlifard}$^{\mathsf{a}}$, {{Sh. Niknejad}$^{\mathsf{b}}$, {R. Nikandish${}^{\mathsf{c}}$\thanks{Corresponding author}}}  }
}\vspace{3mm}\\
{\footnotesize{${}^{\mathsf{b}}$\it Department of Mathematics, Roudbar Branch, Islamic Azad University,}}\\
{\footnotesize{\it   Roudbar, Iran}}\\
{\footnotesize{${}^{\mathsf{b}}$\it Department of Mathematics, Gachsaran Branch, Islamic Azad University,}}\\
{\footnotesize{\it   Gachsaran, Iran}}\\
{\footnotesize{${}^{\mathsf{c}}$\it Department of Mathematics, Jundi-Shapur University of Technology,}}\\
{\footnotesize{\rm P.O. BOX \rm{64615-334}, Dezful, Iran}}\\
{\footnotesize{ $\mathsf{m.adlifard@iauroudbar.ac.ir}$}}\quad\quad
{\footnotesize{$\mathsf{Sh.niknejad@gmail.com}$}}\quad\quad
{\footnotesize{ $\mathsf{r.nikandish@ipm.ir}$}}\\
{\footnotesize{$\mathsf{}$ }}}
\date{}

\begin{document}

\maketitle

\begin{abstract}
{\small The prime ideal sum graph of a commutative unital ring $R$,
denoted by $PIS(R)$, is an undirect and simple graph whose vertices are  non-trivial ideals of $R$ and there exists and edge between to distinct vertices  if and only if  their sum is a prime ideal of $R$. In this paper, the metric dimension of $PIS(R)$ is discussed and some formulae for  this parameter in  $PIS(R)$ are given.

}
\end{abstract}
\begin{center}\section{Introduction}\end{center}
\par
Finding metric dimension in a graph is an example of NP-hard identification problem in discrete structures which has many applications not only in robotic but also in chemistry, image processing, combinatorial optimization and so on. The study of this notion was initiated by Slater \cite{slat} in 1975 and Harary \cite{slat} in 1976 , independently and several graph theorists computed this concept in some classes of graphs, see for instance \cite{[5], Ji, Khuller}. Both of wide range of applications  and complexity of computations  have caused considerable attention in
characterizing this invariant for graphs  associated with algebraic structures, some
examples in this direction may be found in  \cite{dol, dolz, eb2, ma, nik, Pirzadaaz, imranaz, Pirzada1, sha}. This paper is in this theme and aims to investigate the  metric dimension in  prime ideal sum graphs of commutative rings.
\par

Throughout this paper, all rings are assumed to be commutative with identity. The sets of all ideals, maximal ideals and jacobson radical of $R$ are denoted by $A(R)$, $\mathrm{Max}(R)$ and $J(R)$, respectively. For a subset $T$ of a ring $R$ we let $T^*=T\setminus\{0\}$. Moreover, by $A(T)$, we mean the set of ideals of $R$ which are contained in $T$.  The ring $R$ is called \textit{reduced} if $0$ is the only nilpotent element of $R$. Some more definitions about commutative rings can be find in \cite{ati}.

  We use the standard terminology of graphs following \cite{west}.
 Throughout this work, all graphs $G=(V,E)$ are simple and connected. As usual, the symbol  $\mathrm{diam}(G)$ represents the diameter of $G$. The distance between two distinct vertices $x$ and $y$ is denoted by $d(x,y)$. Also, a complete graph of order $n$ is denoted by $K_n$.   For a vertex $x$ in $V(G)$, open neighborhood of $x$ is denoted $N(x)$.
Let $G=(V,E)$ be a  connected graph, $S=\{v_1,v_2,\dots,v_k\}$ be an ordered subset of $V$ and $v\in V(G)\setminus S$.
  The vector $D(v|S)=(d(v,v_1),d(v,v_2),\dots, d(v,v_k))$ is called the \textit{representation of} $v$ with respect to $S$. A set $S\subseteq V$ is called a \textit{resolving set for} $G$ if distinct vertices of $G$ have different representations with respect to $S$.
  A resolving set of  minimum cardinality is called \textit{the metric basis for} $G$   and the cardinality of metric basis is called  \textit{the metric dimension of} $G$. The metric dimension of $G$ is denoted by $dim_M(G)$.

Let $R$ be a commutative ring with identity. The prime ideal sum graph of $R$,
denoted by
$PIS(R)$, is a graph whose vertices are   nonzero  proper  ideals of $R$
 and two distinct vertices $I$ and $J$ are adjacent  if and only if  $I+J$ is a prime ideal of $R$. This graph was first introduced and studied in \cite{sumgraph} and many interesting results were explored by the authors. For instance, they investigated connectedness, diameter, girth, clique number and domination number of  $PIS(R)$. In this paper, we study the metric dimension of $PIS(R)$ and provide some metric dimension formulas for  $PIS(R)$.

 {\begin{center}{\section{ Metric dimension of  a  prime ideal sum graph of a reduced ring}}\end{center}}\vspace{-2mm}

In this section, the exact value of metric dimension in a  prime ideal sum graph associated with a reduced ring $R$ in terms of the number of maximal ideals of $R$ is given. First, we need to state the following result.

Suppose that $R$ is a ring. By \cite[Theorem 2.6]{sumgraph}, $PIS(R)$ is disconnected if and only if $R$ is a direct product of two fields. Thus we exclude such situation from the paper.

\begin{lem}\label{dimfinite}
If $R$ is a ring, then the following statements are equivalent.

$(1)$ $dim_M(PIS(R))$ is finite.

 $(2)$ $R$ has only finitely many ideals.
\end{lem}
\begin{proof}
{$(1)\Rightarrow (2)$  Assume that $dim_M(PIS(R))$ is finite and  $W=\{I_1,I_2,\dots,I_n\}$ is a metric basis in $PIS(R)$,  for some non-negative integer $n$. Let $I\in V(PIS(R))\setminus W$. Then the number of all possible ways for $D(I|W)$ equals $(4+1)^n$, as $diam(PIS(R))\leq 4$ (see  \cite[Theorem 2.6]{sumgraph}). Thus $|V(PIS(R))|\leq 5^n+n$ and hence  $R$ has finitely many ideals.

$(2)\Rightarrow (1)$ is clear.
}
\end{proof}

 In order to compute $dim_M(PIS(R))$, it is enough to consider Artinian rings $R$ (Indeed, rings with finitely many ideals), by Lemma \ref{dimfinite}. Moreover, every reduced ring with finitely many ideals is a direct product of finitely many fields, by \cite[Theorem 8.7]{ati}. Using these facts, we prove the main result of this section.

\begin{thm}\label{dimprod}
Suppose that $R$ is a reduced ring and  $dim_M(PIS(R))$ is finite. Then the following statements hold:

$(1)$ If $|\mathrm{Max}(R)|=3$, then  $dim_M(PIS(R))=2$.

$(2)$
If $|\mathrm{Max}(R)|=n\geq 4$, then $dim_M(PIS(R))=n$.
\end{thm}
\begin{proof}
{$(1)$
If $n=3$, then  $R\cong F_1\times F_2\times F_3$, where $F_i$ is a  field for every $1\leq i\leq 3$. Let
 $W=\{(0)\times F_2\times F_3, F_1\times (0)\times F_3\}$. Then by Figure \ref{figure:f1},

$D({(0)\times F_2\times (0)}|W)=(1,2)$,

$D({F_1\times (0)\times (0)}|W)=(2,1)$,

$D({(0)\times (0)\times F_3}|W)=(1,1)$,

$D({F_1\times F_2\times (0)}|W)=(2,2)$.

Thus, $D(I|W)\neq D(J|W)$,  for every $I,J\in V(PIS(R))\setminus W$. Therefore, $dim_M(PIS(R))=2$.
\unitlength=1.5mm
\begin{figure}[htb]
 \centering
\begin{picture}(60,30)(20,-10)
\put (48,12){\circle*{1.2}}
\put (43,14){\tiny{$(0)\times F_2\times F_3$}}
\put (48,6){\line (0,1){6}}
\put (48,6){\circle*{1.2}}
\put (37,7){\tiny{$(0)\times F_2\times (0)$}}
\put (42,0){\circle*{1.2}}
\put (30,1){\tiny{$F_1\times (0)\times (0)$}}
\put (42,0){\line (1,1){6}}
\put (42,0){\line (1,0){12}}
\put (42,0){\line (1,0){12}}
\put (54,0){\circle*{1.2}}
\put (55,0){\tiny{$(0)\times (0)\times F_3$}}
\put (54,0){\line (-1,1){6}}
\put (54,0){\line (-1,2){6}}
\put (36,-3){\circle*{1.2}}
\put (31,-6){\tiny{$F_1\times F_2\times (0)$}}
\put (36,-3){\line (2,1){6}}
\put (36,-3){\line (4,3){12}}
\put (60,-3){\circle*{1.2}}
\put (55,-6){\tiny{$F_1\times (0)\times F_3$}}
\put (60,-3){\line (-2,1){6}}
\put (60,-3){\line (-6,1){18}}
\end{picture}
 \caption{\rm ${PIS}(F_1\times F_2\times F_3) $} \label{figure:f1}
\end{figure}
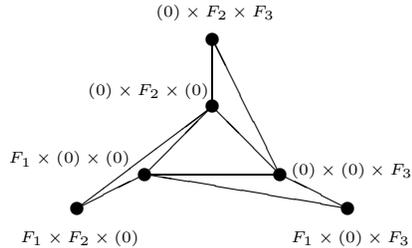

$(2)$ Assume that $n\geq 4$. It follows from Lemma \ref{dimfinite} that $R\cong F_1\times\cdots \times F_n$, where  $F_i$ is a field for every $1\leq i\leq n$.
We show that $dim_M(PIS(R))=n$. Indeed, we have the following claims:

\textbf{Claim 1.}
$dim_M(PIS(R))\geq n$.

By Lemma \ref{dimfinite}, we may suppose that $dim_M(PIS(R))=k$, where $k$ is a positive integer. Since $R$ is a principal ideal ring,  by
  \cite[Theorem 2.6]{sumgraph}, $diam(PIS(R))\in \{1,2\}$ and so there are  $2^k$ possible ways for $D(I|W)$, if $I\in V(PIS(R))\setminus W$. This implies that $|V(PIS(R))|-k\leq 2^k$. It is clear that $|V(PIS(R))|=2^n-2$ and thus $2^n-2-k\leq 2^k$. Hence  $2^n\leq 2^k+2+k$. Now,
$n\geq 4$ implies that $k\geq n$. Therefore, $dim_M(PIS(R))\geq n$.

\textbf{Claim 2.}
$dim_M(PIS(R))\leq n$.

For every $1\leq i\leq n$, let $\mathfrak{m}_i=(I_1,I_2,\dots,I_n)\in A(R)^*$  such that $I_i=0$ and $I_j=F_j$, for every $1\leq j\leq n$ with $i\neq j$. Put
$W=\{\mathfrak{m}_1,\mathfrak{m}_2,\dots,\mathfrak{m}_n\}$ (Obviously,  $W=\mathrm{Max}(R)$).
We prove that $W$ is a resolving set for $PIS(R)$. Let $I,J\in V(PIS(R))\setminus W$  and $I\neq J$. We have to show that  $D(I|W)\neq D(J|W)$. Let $I=(I_1,I_2,\dots,I_n)$ and $J=(J_1,J_2,\dots,J_n)$. Since $I\neq J$, there exists $1\leq i\leq n$ such that either $I_i=0$ and $J_i=F_i$ or $I_i=F_i$ and $J_i=0$. Without loss of generality, assume that $I_1=0$ and $J_1=F_1$. It is not hard to check that
$d(I,\mathfrak{m}_1)=1$ and $d(J,\mathfrak{m}_1)=2$. This clearly shows that $D(I|W)\neq D(J|W)$, as desired.

 Now, Claims 1, 2 show that $dim_M(PIS(R))=n$, if $n\geq 4$.
}
\end{proof}

{\begin{center}{\section{ Metric dimension of  a  prime ideal sum graph  of a  non-reduced ring}}\end{center}}\vspace{-2mm}

In this section, $dim_M(PIS(R))$ is investigated, in case $R$ is non-reduced.

We begin with the following lemma which provides an upper bound for $dim_M(PIS(R))$, when $R$ is a direct product of finitely many rings with one non-trivial ideal.

\begin{lem}\label{lem2n}
 Suppose that $R\cong R_1\times\cdots \times R_n$ and every $R_i$ is a ring with a unique non-trivial ideal $\mathfrak{m}_i$,  for $1\leq i\leq n$. Then
   $dim_M(PIS(R))\leq2n$.
\end{lem}
\begin{proof}
{ Suppose that

 $A=\{(0,R_2,\dots,R_n), (R_1,0,R_3,\dots,R_n),\dots,(R_1,\dots,R_{n-1},0)\}$,

 $B=\{(\mathfrak{m}_1,R_2,\dots,R_n), (R_1,\mathfrak{m}_2,R_3,\dots,R_n),\dots,(R_1,\dots,R_{n-1},\mathfrak{m}_n)\}$  (in fact $B={\rm Max}(R)$) and

$W=A\cup B$.

 It is proved that $W$ is a resolving set in $PIS(R)$. Let $I=(I_1,\dots,I_n),J=(J_1,\dots,J_n)\in A(R_i)^*\setminus W$. It is enough to show that
$D(I|W)\neq D(J|W)$.

We have to consider the following cases:

\textbf{Case 1.}
$I\subseteq J(R)$. If $J\subseteq J(R)$, too, then with no loss of generality, one may assume that
$I_1=0$ and $J_1=\mathfrak{m}_1$, as $I\neq J$. We have $D(I|W)=(2,\underbrace{-,\cdots,-}_{n-1},\underbrace{1,\cdots,1}_n)$
and $D(J|W)=(1,\underbrace{-,\cdots,-}_{n-1},\underbrace{1,\cdots,1}_n)$.
Therefore, $D(I|W)\neq D(J|W)$. If  $J\nsubseteq J(R)$, then with no loss of generality, one may assume that
 $J_n=R_n$. Thus $D(I|W)=(\underbrace{-,\cdots,-}_{n},\underbrace{1,\cdots,1}_n)$
and $D(J|W)=(\underbrace{-,\cdots,-}_{n-1},2,\underbrace{-,\cdots,-}_{n-1},2)$.
Therefore, $D(I|W)\neq D(J|W)$.

\textbf{Case 2.} $I\nsubseteq J(R)$ and $J\nsubseteq J(R)$.
In this case, $I_i=R_i$ and $J_j =R_j$, for some $1\leq i,j\leq n$.
If for every $1\leq i\leq n$, $I_i=R_i$ if and only if $J_i =R_i$, then since
  $I\neq J$,  for some $1\leq i\leq n$, $\{I_i,J_i\}\subseteq \{0,\mathfrak{m}_i\}$. Without lost of generality, we may assume that $I_1=0$ and $J_1 =\mathfrak{m}_1$.
  Now, we have $D(I|W)=(2,\underbrace{-,\cdots,-}_{n-1},\underbrace{x_1,\cdots,x_n}_n)\neq (1,\underbrace{-,\cdots,-}_{n-1},\underbrace{x_1,\cdots,x_n}_n)=D(J|W)$ such that $x_i\in \{1,2\}$ and there exists
  $1\leq i\leq n$ such that  $x_i=2$.

  If  $I_i=R_i$ and  $J_i \neq R_i$, for some $1\leq i\leq n$, then $\{J_i\}\subseteq \{0,\mathfrak{m}_i\}$. Without lost of generality, assume that $I_1=R_1$ and $J_1 =0$ or $J_1 =\mathfrak{m}_1$.
   Thus $D(I|W)=(2,\underbrace{-,\cdots,-}_{n-1},2,\underbrace{x_2,\cdots,x_n}_{n-1})$, but  $D(J|W) =(2,\underbrace{-,\cdots,-}_{n-1},1,\underbrace{x_2,\cdots,x_{n-1}}_{n-1})$, if  $J_1 =0$ and
   $D(J|W) =(1,\underbrace{-,\cdots,-}_{n-1},1,\underbrace{x_2,\cdots,x_{n-1}}_{n-1})$, if  $J_1 =\mathfrak{m}_1$.
  The above arguments show that  $D(I|W)\neq D(J|W)$.
   Finally,  we note that the situation $I\nsubseteq J(R)$ and $J\subseteq J(R)$ is similar to  $J\nsubseteq J(R)$ and $I\subseteq J(R)$ in Case 1.
Therefore, $W$ is a resolving set for
  $PIS(R)$. This implies that
$dim_M(PIS(R))\leq |W|= 2n$.
 }
\end{proof}

The following example shows that the upper bound $2n$ in Lemma \ref{lem2n} is sharp.
\begin{example}\label{11001}
\end{example}
{$(1)$ \rm Let $n=2$ in Lemma \ref{lem2n} and
$W=\{(R_1, 0),(0, R_2),(R_1,\mathfrak{m}_2)\}$.
Then
$D({(\mathfrak{m}_1,R_2)}|W)=(2,1,2)$,
$D({(0,\mathfrak{m}_2)}|W)=(1,2,1)$,
$D({(\mathfrak{m}_1,0)}|W)=(2,1,1)$,
$D({(\mathfrak{m}_1,\mathfrak{m}_2)}|W)=(1,1,1)$
 and hence $W$ is a resolving set for $PIS(R)$. This means that    $dim_M(PIS(R)\leq 3$. On the other hand, since
 $diam(PIS(R))=2$ and $|V(PIS(R))|=7$, we conclude that $dim_M(PIS(R)\geq 3$. So
   $dim_M(PIS(R))=3$.

$(2)$  Let $n=3$ in Lemma \ref{lem2n}.
 By Lemma \ref{lem2n},
   $dim_M(PIS(R))\leq 6$. But we show that   $dim_M(PIS(R))= 5$. Since $|V(PIS(R)))|=25$ and  $diam(PIS(R))=2$, we deduce that $dim_M(PIS(R))\geq 5$.
   To complete the proof,  let

$W=\{(0, R_2, R_3),(R_1,0, R_3),(R_1,R_2,0), (R_1,\mathfrak{m}_2,\mathfrak{m}_3),(\mathfrak{m}_1,\mathfrak{m}_2,R_3)\}$.

 Then we have

  $D({(\mathfrak{m}_1,\mathfrak{m}_2,\mathfrak{m}_3)}|W)=(1,1,1,2,2)$,

    $D({(\mathfrak{m}_1,\mathfrak{m}_2,0)}|W)=(1,1,2,2,2)$,

    $D({(\mathfrak{m}_1,0,\mathfrak{m}_3)}|W)=(1,2,1,2,2)$,

    $D({(0,\mathfrak{m}_2,\mathfrak{m}_3)}|W)=(2,1,1,2,2)$,

    $D({(\mathfrak{m}_1,0,0)}|W)=(1,2,2,2,2)$,

    $D({(0,\mathfrak{m}_2,0)}|W)=(2,1,2,2,2)$,

     $D({(0,0,\mathfrak{m}_3)}|W)=(2,2,1,2,2)$,

     $D({(R_1,0,0)}|W)=(2,2,2,2,1)$,

   $D({(0,R_2,0)}|W)=(2,2,2,1,1)$,

   $D({(0,0,R_3)}|W)=(2,2,2,1,2)$,

    $D({(\mathfrak{m}_1,R_2,R_3)}|W)=(1,2,2,2,1)$,

    $D({(R_1,\mathfrak{m}_2,R_3)}|W)=(2,1,2,1,1)$,

    $D({(R_1,R_2,\mathfrak{m}_3)}|W)=(2,2,1,1,2)$,

      $D({(\mathfrak{m}_1,R_2,\mathfrak{m}_3)}|W)=(1,2,1,1,1)$,

      $D({(R_1,\mathfrak{m}_2,0)}|W)=(2,1,2,2,1)$,

      $D({(R_1,0,\mathfrak{m}_3)}|W)=(2,2,1,2,1)$,

      $D({(\mathfrak{m}_1,R_2,0)}|W)=(1,2,2,1,1)$,

      $D({(0,R_2,\mathfrak{m}_3)}|W)=(2,2,1,1,1)$,

      $D({(0,\mathfrak{m}_2,R_3)}|W)=(2,1,2,1,2)$,

      $D({(\mathfrak{m}_1,0,R_3)}|W)=(1,2,2,1,2).$

      It is not hard to see that  $D(I|W)\neq D(J|W)$, for all $I,J \in A(R)^*\setminus W$. Thus $W$ is a resolving set in $PIS(R)$ and so $dim_M(PIS(R))= 5$.
}

 Lemma \ref{lem2n} enables us  to give the exact value of $dim_M(PIS(R))$, when $R$ is a direct product of finitely many rings with exactly one non-trivial ideal.

\begin{thm}\label{isomorphismq}
  Suppose that $R\cong R_1\times\cdots \times R_n$ and every $R_i$ is a ring with a unique non-trivial ideal $\mathfrak{m}_i$,  for $1\leq i\leq n$. Then the following statements hold:

 $(1)$ If $n=1$, then    $dim_M(PIS(R))=0$.

 $(2)$ If $n=2$ or $n=3$,  then    $dim_M(PIS(R))=2n-1$.

  $(3)$ If $n\geq 4$, then
   $dim_M(PIS(R))=2n$.
\end{thm}
\begin{proof}
{$(1)$ If $n=1$, then the result is obvious.

 $(2)$ If $n=2$ or $n=3$,  then  the result follows from Example \ref{11001}.

$(3)$ Let $n\geq 4$,

$A=\{(0,R_2,\dots,R_n), (R_1,0,R_3,\dots,R_n),\dots,(R_1,\dots,R_{n-1},0)\}$,

 $B=\{(\mathfrak{m}_1,R_2,\dots,R_n), (R_1,\mathfrak{m}_2,R_3,\dots,R_n),\dots,(R_1,\dots,R_{n-1},\mathfrak{m}_n)\}$ and

$C=A(R_i)^*\setminus A\cup B$.

We claim that $A\cup B\subseteq W$, where
$W$ is  a resolving set for  $PIS(R)$.
First, we show that $A\subseteq W$ (Indeed, we are going to prove that all ideals contained in $J(R)\setminus W$ are resolved only by (some) vertices contained in $A$). Let $|A(J(R)) \cap W|=t$ and $I\in A(J(R)) \cap W$. Since  $|A(J(R))|=2^n-1$ and  $d(I,K)=1$ for every $K \in B$ and $d(I,K)=2$ for every $K \in C$ with $I\neq K$, we must have $2^n-1-t \leq 2^{|S|}$, where $S \subseteq A$  is a resolving set for $A(J(R)\setminus W)$. If $S\neq A$, then since  $|S| \leq n-1$, we infer that $2^n-1-t \leq 2^{n-1}$, a contradiction (as $n \geq 4$). Hence $A \subseteq W$. Next, we show that $B\subseteq W$. Let
$T=\{(I_1,\dots, I_n)\,\,|\,\, I_i \in  \{0,R_i \}\}$. Since $d(I,J)=2$, for every $I \in T$ and $J\in A$, we must add some elements of $A(R)\setminus A$ to $W$ for resolving  elements in $T\setminus W$.
 Since $|T|=2^n-2$, we must add at last $n $ elements of $A(R)\setminus A$ to $W$(In fact we must add at last $t$ elements till $|T\setminus A|\leq 2^t$). This proves the claim and so   $dim_M(PIS(R))= |W| \geq 2n$ . Now, the result follows from Lemma \ref{lem2n}.
 }
\end{proof}

The following remark has a key role in proof of Theorem \ref{isomorphismqq}.

\begin{remark}\label{isomorphisms}
For a connected graph $G$, if
$V_1,V_2,\dots,V_k$ is a partition of $V(G)$ such that for every $1\leq i\leq k$,  $x,y\in V_i$ implies that  $N(x)=N(y)$.   Then  $dim_M(G)\geq |V(G)|-k$.
\end{remark}

\begin{thm}\label{isomorphismqq}
 Suppose that $R\cong R_1\times\cdots \times R_n$, where $(R_i, \mathfrak{m}_i)$ is an Artinian local principal ideal ring and $|A(R_i)^*|\geq 2$, for every $1\leq i\leq n$. Then the following statements hold:

  $(1)$ If either $n\neq 1$ or  $n=1$ and $|A(R_i)^*|\geq 3$, then   $dim_M(PIS(R))=|A(R)^*|-3^{n}+1$.

   $(2)$ If $n=1$ and $|A(R_i)^*|=2$, then   $dim_M(PIS(R))=1$.
\end{thm}
\begin{proof}
{Suppose that  $I= (I_1,\dots, I_n)$ and $J= (J_1,\dots,J_n)$ are vertices of $PIS(R)$.
  Define the relation  $\thicksim$ on  $V(PIS(R))$ as follows: $X\thicksim Y$, whenever for every $1\leq i\leq n$, the following conditions hold.

 $(1)$``$I_i=\mathfrak{m}_i$ if and only if $ J_i=\mathfrak{m}_i$''.

 $(2)$``$I_i\subseteq \mathfrak{m}_i$ with $I_i\neq \mathfrak{m}_i$ if and only if $J_i\subseteq \mathfrak{m}_i$ with $J_i\neq \mathfrak{m}_i$.

  $(3)$``$I_i=R_i$ if and only if $J_i=R_i$''.

  Clearly, $\thicksim$ is an equivalence relation on $V(PIS(R))$.
   The equivalence class of $I$ is denoted by
   $[I]$. Suppose that
  $X$ and $Y$ are two elements of  the equivalence class of $I$. By definition of the relation  $\thicksim$ and adjacency condition
for   $X$ and $Y$,
 one can check that $N(X)=N(Y)$.
 Now, Remark \ref{isomorphisms} and the fact that  ``the number of equivalence classes is $3^n-1$" imply that $dim_M(PIS(R))\geq|A(R)^*|-(3^{n}-1)=|A(R)^*|-3^{n}+1$.

 To complete the proof, it is shown that $dim_M(PIS(R))\leq |A(R)^*|-3^{n}+1$.
  Let $ (K_1,\dots, K_n)\in V(PIS(R))$ be a fixed vertex such that   for every $1\leq i\leq n$, $0\neq K_i\subset \mathfrak{m}_i$ and
  $$A=\{(I_1,\dots,I_n)\in V(PIS(R))|\,\, I_i\in \{0,R_i,\mathfrak{m}_i\}\,\,\mathrm{for\,\, every}\,\,1\leq i\leq n\}\cup
\{(K_1,\dots, K_n)\}.$$

 We claim that
$W=A(R)^*\setminus A$  is a resolving set for $PIS(R)$. Let $I,J\in A$  and $I\neq J$. It is enough to show that  $D(I|W)\neq D(J|W)$.
Consider the following sets:

  $V_1=\{(I_1,\dots,I_n)\in V(PIS(R))|\,\, I_i\in \{0,R_i\}\,\,\mathrm{for\,\, every}\,\,1\leq i\leq n\}$,

 $V_2=\{(I_1,\dots,I_n)\in V(PIS(R))|\,\, I_i\in \{0,\mathfrak{m}_i\}\,\,\mathrm{for\,\, every}\,\,1\leq i\leq n\}$,

$V_3=A\setminus \{V_1\cup V_2\}$,

  $T_1=\{(K_1,R_2\dots,R_n), (R_1,K_2,R_2\dots,R_n),\dots,(R_1,\dots,R_{n-1},K_n)\}$,

 $T_2=\{(I_1,\dots,I_n)\in V(PIS(R))|\,\, I_i=K_i,I_j=\mathfrak{m}_j\,\mathrm{with\,\, i\neq j} \,\mathrm{and \,\,for\,\, every}\,k\in \{1,\dots, n\}\setminus \{i,j\}, I_k=R_k\}$ and

$T_3=W\setminus \{T_1\cup T_2\}$.

Arrange the elements of $W$ as follows:

$$W=\{\underbrace{w_1,\dots,w_{|T_1|}}_{\in T_1},\underbrace{w_{|T_1|+1},\dots,w_{|T_1|+|T_2|}}_{\in T_2},\underbrace{w_{|T_1|+|T_2|+1},\dots,w_{|W|}}_{\in T_3}\}.$$

If $I\in V_1$, then

$$D(I|W)=(\underbrace{2,\dots,2}_{|T_1|},\underbrace{y_1,\dots,y_{|T_2|}}_{|T_2|},\underbrace{z_1,\dots,z_{|T_3|}}_{|T_3|}),$$

 such that $y_i\in \{1,2\}$ and  $y_i\neq y_j$, for some
  $1\leq i,j\leq |T_2|$ with $i\neq j$.

If $I\in V_2$, then

$$D(I|W)=(\underbrace{x_1,\dots,x_{|T_1|}}_{|T_1|},\underbrace{2,\dots,2}_{|T_2|},\underbrace{z_1,\dots,z_{|T_3|}}_{|T_3|}),$$

 such that $x_i\in \{1,2\}$ and  $x_i\neq x_j$, for some
  $1\leq i,j\leq |T_1|$ with $i\neq j$.

  Finally, if $I\in V_3$, then

  $$D(I|W)=(\underbrace{x_1,\dots,x_{|T_1|}}_{|T_1|},\underbrace{y_1,\dots,y_{|T_2|}}_{|T_2|},\underbrace{z_1,\dots,z_{|T_3|}}_{|T_3|}),$$

   such that $x_i=y_j=1$, where
  $1\leq i\leq |T_1|$ and $1\leq j\leq |T_2|$.

  The above vectors shows that $D(I|W)\neq D(J|W)$, if $I\in V_i$ and $J\in V_j$, for every   $1\leq i,j\leq 3$ with $i\neq j$. Next, we  show that $D(I|W)\neq D(J|W)$, if $I,J\in V_i$, for every   $1\leq i\leq 3$.
   The following cases should be considered:

\textbf{Case 1.} $I,J\in V_1$.
Since $I\neq J$, with no loss of generality, one may assume that $I_i=0$ and $J_i=R_i$, for some $1\leq i\leq n$. Moreover, $J_i=R_i$ implies that the zero element must appear among components of $J$. Let $J_k=0$ ($J_k$ is the $k$-th component of $J$), for some $k\neq i$.   Consider $y=(y_1,\dots,y_n)\in T_2$  such that
$y_i=K_i$ and $y_k=\mathfrak{m}_k$. Moreover, the other components of $y$ are  $R_j$'s, $j\neq i,k$.
Now, $d(I,y)=2$ and $d(J,y)=1$ and so $D(I|W)\neq D(J|W)$.

\textbf{Case 2.} $I,J\in V_2$.
Since $I\neq J$, without lost of generality, we may assume that $I_i=0$ and $J_i =\mathfrak{m}_i$, for some $1\leq i\leq n$.
Consider $x=(x_1,\dots,x_n)\in T_1$  such that
$x_i=K_i$. Moreover, the other components of $x$ are $R_j$'s, $j\neq i$.
Since $d(I,x)=2\neq 1=d(J,x)$, $D(I|W)\neq D(J|W)$.

\textbf{Case 3.} $I,J\in V_3$.
  The following subcases may occur:

\textbf{Subcase 1.}  $I_i=\mathfrak{m}_i$ and $J_i=0$, for some $1\leq i\leq n$.
  Since $I_i=\mathfrak{m}_i$, for some $1\leq j\leq n$ with $i\neq j$, $I_j=R_j$.
  Let
  $y=(y_1,\dots,y_n)\in T_3$  such that
$y_i=K_i$ and $y_j=K_j$. Moreover, the other components of $y$ are $R_k$'s, $k\neq i,j$.
Now, $d(I,y)=1$ and $d(J,y)=2$ and so $D(I|W)\neq D(J|W)$.

\textbf{Subcase 2.}
 $I_i=\mathfrak{m}_i$ and $J_i=R_i$, for some $1\leq i\leq n$.
  Since $J_i=R_i$, for some $1\leq j\leq n$ with $i\neq j$, $J_j=\mathfrak{m}_j$.
  If $I_j=\mathfrak{m}_j$, then let
  $y=(y_1,\dots,y_n)\in T_2$  such that
$y_i=K_i$ and $y_j=\mathfrak{m}_j$. Moreover, the other components of $y$ are  $R_k$'s, $k\neq i,j$.
Now, $d(I,y)=2$ and $d(J,y)=1$ and so $D(I|W)\neq D(J|W)$.
If
$I_j=R_j$, then let
  $y=(y_1,\dots,y_n)\in T_1$  such that
$y_j=K_j$ and the other components of $y$ are  $R_i$'s, $i\neq j$.
Now, $d(I,y)=2$ and $d(J,y)=1$ and so $D(I|W)\neq D(J|W)$.

\textbf{Subcase 3.} $I_i=\mathfrak{m}_i$ if and only if
  $J_i=\mathfrak{m}_i$, for every $1\leq i\leq n$.
Thus there exists $1\leq j\leq n$ such that $I_j=R_j$ and $J_j=0$.
We note that  there exists $k\neq i,j$ $J$ such that $J_k=R_k$ with. Let
  $y=(y_1,\dots,y_n)\in T_3$  such that
$y_i=K_i$, $y_j=0$ and  the other components of $y$ are equal $R_k$'s, $k\neq i,j$.
Now, $d(I,y)=1$ and $d(J,y)=2$ and so $D(I|W)\neq D(J|W)$.
}
\end{proof}

The following corollary is an immediate consequence of Theorem \ref{isomorphismqq}.
\begin{cor}\label{isomorphismqd}
 Suppose that $R\cong R_1\times\cdots \times R_n$, where $R_i$ is an Artinian local ring such that  for every $1\leq i\leq n$, $|A(R_i)^*|= 2$. Then
   $dim_M(PIS(R))=4^n-3^n-1$.
\end{cor}

\begin{cor}\label{isomorphismqwen}
Let $R\cong S_1\times S_2\times S_3$, $n,m,k$ are non-negative integer with  $n,m,k\not\in \{1,2,3 \}$, $S_1\cong R_1\times\cdots \times R_n$, $S_2\cong R^{\prime}_1\times\cdots \times R^{\prime}_m$ and $S_3\cong F_1\times \cdots \times F_k$,  where $(R_i, \mathfrak{m}_i)$ is an Artinian local principal ideal ring such that  for every $1\leq i\leq n$, $|A(R_i)^*|\geq 2$, $(R^{\prime}_i, \mathfrak{m}^{\prime}_i)$ is an  Artinian local ring such that  $|A(R^{\prime}_i)^*|= 1$, for every $1\leq i\leq m$ and $F_i$ is a field, for every $1\leq i\leq k$. Then
   $dim_M(PIS(R))=|A(R)^*|-3^{n+m}2^k+2m+k+2$.
\end{cor}
\begin{proof}
{
Assume that 

  $A=\{(I_1,\dots,I_{n+m+k})\in V(PIS(R))\,\,|\,\, I_i\in \{0,R_i,\mathfrak{m}_i,R^{\prime}_j,\mathfrak{m}^{\prime}_j,F_t\}\,\,\mathrm{for\,\, every}\,\,1\leq i\leq n,\,\,n+1\leq j\leq n+m,\,\,n+m+1\leq t\leq n+m+k\}$,

 $B=\{(R_1,\dots,R_n,0,R^{\prime}_2,\dots,R^{\prime}_m,F_1,\dots,F_k), (R_1,\dots,R_n,R^{\prime}_1,0,R^{\prime}_3,
 \dots,R^{\prime}_m,F_1,\dots,F_k),$ $\dots,(R_1,\dots,R_n,
 R^{\prime}_1,R^{\prime}_2,\dots,R^{\prime}_{m-1},0,F_1,\dots,F_k)\}$,

  $C=\{(R_1,\dots,R_n,\mathfrak{m}^{\prime}_1,R^{\prime}_2,\dots,R^{\prime}_m,F_1,\dots,F_k), (R_1,\dots,R_n,R^{\prime}_1,\mathfrak{m}^{\prime}_2,R^{\prime}_3,
 \dots,R^{\prime}_m,F_1,\dots,F_k),$

 $\dots,(R_1,\dots,R_n,
 R^{\prime}_1,R^{\prime}_2,\dots,R^{\prime}_{m-1},\mathfrak{m}^{\prime}_m,F_1,\dots,F_k)\}$,

$V_1=A(R)^*\setminus A$,

$V_2=B\cup C$ and

   $V_3=\{(R_1,\dots,R_n,R^{\prime}_1,\dots,R^{\prime}_m,0,F_2,\dots,F_k), (R_1,\dots,R_n,R^{\prime}_1,
 \dots,R^{\prime}_m,F_1,0,F_2,\dots,F_k),$
 
 $\dots,(R_1,\dots,R_n,
 R^{\prime}_1,\dots,R^{\prime}_m,F_1,\dots,F_{k-1},0)\}.$

Arrange the elements of $W$ as follows:

$W=\{\underbrace{w_1,\dots,w_{|V_1|}}_{\in V_1},\underbrace{w_{|V_1|+1},\dots,w_{|V_1|+|V_2|}}_{\in V_2},\underbrace{w_{|V_1|+|V_2|+1},\dots,w_{|V_1|+|V_2|+|V_3|}}_{\in V_3}\}.$

 By similar proofs of Theorems \ref{isomorphismq}, \ref{isomorphismqq} and \ref{dimprod}, $W$ is a
 metric basis for $PIS(R)$ and so
    $dim_M(PIS(R))=|W|=|V_1|+|V_2|+|V_3|=|A(R)^*|-3^{n+m}2^k+2m+k+2$.
 }
\end{proof}

We close this paper with the following example which shows that the condition $n,m,k\not\in \{1,2,3 \}$  is required in  Corollary \ref{isomorphismqwen}.
\begin{example}
\end{example}

{\rm $(1)$  Let $R=\mathbb{Z}_4\times  \mathbb{Z}_2$.
By the following figure, we have
$dim_M(PIS(R))=2$.
But $n=0$ and $m=k=1$ do not satisfy in Corollary \ref{isomorphismqwen}.
\hspace{-8cm}
\unitlength=1.5mm
\begin{picture}(60,30)(-30,-20)
\put (8,4){\circle*{1.2}}
\put (9,4){\tiny{$(0)\times \mathbb{Z}_2$}}
\put (8,4){\line (2,-3){8}}
\put (8,4){\line (-2,-3){8}}
\put (0,-8){\circle*{1.2}}
\put (-8,-6){\tiny{$<2>\times \mathbb{Z}_2$}}
\put (0,-8){\line (1,0){16}}
\put (16,-8){\circle*{1.2}}
\put (15,-6){\tiny{$<2>\times (0)$}}
\put (16,-8){\line (1,0){16}}
\put (32,-8){\circle*{1.2}}
\put (29,-6){\tiny{$\mathbb{Z}_4\times (0)$}}

\put (10,-15){$  PIS(\mathbb{Z}_4\times  \mathbb{Z}_2) $}

\end{picture}

\hspace{8cm}

 $(2)$ Let $R=  \mathbb{Z}_4\times  \mathbb{Z}_2\times  \mathbb{Z}_2$ and
$W=\{(0, \mathbb{Z}_2,  \mathbb{Z}_2),( \mathbb{Z}_4,0,  \mathbb{Z}_2),(\mathbb{Z}_4,  \mathbb{Z}_2,0)\}$.
 Then

  $D({(\mathfrak{m}_1,\mathbb{Z}_2,\mathbb{Z}_2)}|W)=(1,2,2)$,

    $D({(\mathfrak{m}_1,0,\mathbb{Z}_2)}|W)=(1,1,2)$,

    $D({(\mathfrak{m}_1,\mathbb{Z}_2,0)}|W)=(1,2,1)$,

    $D({(\mathfrak{m}_1,0,0)}|W)=(1,1,1)$,

    $D({(0,0,\mathbb{Z}_2)}|W)=(2,1,2)$,

    $D({(0,\mathbb{Z}_2,0)}|W)=(2,2,1)$,

     $D({(\mathbb{Z}_4,0,0)}|W)=(2,1,1).$

     Then  $dim_M(PIS(R))=3$. But  $n=0$, $m=1$ and $k=2$ do not satisfy in Corollary \ref{isomorphismqwen}.

{}

\end{document}